\documentclass[11pt]{article}

\usepackage[a4paper,margin=1in]{geometry}
\usepackage{amsmath,amssymb,amsthm,amsfonts,mathtools}
\usepackage{mathrsfs}
\usepackage{enumitem}
\usepackage{hyperref}
\usepackage[nameinlink,capitalise]{cleveref}
\usepackage{microtype}

\allowdisplaybreaks

\title{Finiteness of Equidistant Affine Subspaces}
\author{Jia Li}
\date{August 12, 2026}

\newtheorem{theorem}{Theorem}[section]
\newtheorem{proposition}[theorem]{Proposition}
\newtheorem{lemma}[theorem]{Lemma}
\newtheorem{corollary}[theorem]{Corollary}

\newtheorem{question}[theorem]{Question}

\theoremstyle{definition}
\newtheorem{definition}[theorem]{Definition}
\newtheorem{remark}[theorem]{Remark}
\newtheorem{example}[theorem]{Example}

\newcommand{\R}{\mathbb R}

\newcommand{\Gr}{\operatorname{Gr}}
\newcommand{\Graff}{\operatorname{Graff}}
\newcommand{\dist}{\operatorname{dist}}
\newcommand{\rank}{\operatorname{rank}}
\newcommand{\Span}{\operatorname{span}}
\newcommand{\col}{\operatorname{col}}
\newcommand{\Mat}{\operatorname{Mat}}

\begin{document}

\maketitle

\begin{abstract}
Let \(0\leqslant k<n\), and let \(\mathcal N(k,n)\) denote the
supremum of the cardinalities of finite families of affine
\(k\)-planes in \(\R^n\) whose pairwise Euclidean separation
distances are all equal to one. For \(k=1\), the finiteness of
\(\mathcal N(1,n)\) was recently established by Solymosi and Zahl.
We extend this result to affine subspaces of arbitrary dimension and
prove the fully explicit bound
\[
\mathcal N(k,n)
\leqslant
2^{\,2^{\,10(k+1)(n-k)}}.
\]

The main difficulty in higher dimensions is that the rational distance
formula degenerates along several possible rank strata of
\(\dim(U+V)\), rather than only along the parallel locus. For each
fixed rank \(r\), we replace the choice of a nonvanishing Gram minor
by a global sum-of-squares expression involving all minors. This
produces a polynomial relation \(H_r(x,y)=0\) which encodes both the
unit-distance equation and an upper-rank condition, while also
vanishing on the diagonal.

We propagate this relation by means of the restricted Zariski closure
and decompose the resulting algebraic set into an explicitly bounded number of connected Nash submanifolds. An analytic diagonal
approximation argument shows that, within each Nash piece, the maximal
possible value of \(\dim(U_x+U_y)\) must decrease. Iterating this rank
descent over all possible ranks and combining the finitely many
standard affine Grassmann charts yields the stated double-exponential
upper bound.
\end{abstract}




\section{Introduction}

Let $0\leqslant k<n$. An affine $k$-plane in $\R^n$ is a set of the form
\[
L=a+U,
\]
where $U\subset\R^n$ is a $k$-dimensional linear subspace and $a\in\R^n$. For two affine $k$-planes $L_1,L_2$, define
\[
d(L_1,L_2)
:=
\inf_{x\in L_1,\,y\in L_2}\|x-y\|.
\]

Let $\Graff(k,n)$ denote the moduli space of affine $k$-planes in $\R^n$. We are interested in finite families of affine \(k\)-planes whose pairwise
distances are all equal to one.

\begin{question}[Generalized Littlewood seven-cylinder problem]
Let \(0\leqslant k<n\). Define
\[
\mathcal{N}(k,n)
:=
\sup
\left\{
|\mathcal{L}|
\;\middle|\;
\begin{array}{l}
\mathcal{L}=\{L_1,\cdots,L_m\}\subset\Graff(k,n),\text{with}\ d(L_i,L_j)=1,\\[2mm]
\text{for all distinct}\ i,j\in\{1,2,\cdots,m\}.\end{array}
\right\}
\in
\mathbb{Z}_{\geqslant 0}\cup\{+\infty\}.
\]
Determine the value of \(\mathcal{N}(k,n)\).
\end{question}

This problem may be viewed as a natural higher-dimensional analogue of the
classical equilateral set problem. The case \((k,n)=(1,3)\) is the classical Littlewood seven-cylinder problem: it asks for the maximum number of lines in \(\R^3\) whose pairwise distances are all equal to \(1\). Equivalently, it asks for the maximum number of congruent infinite cylinders that can be arranged in three-dimensional Euclidean space so that every pair touches.

The known three-dimensional line case has recently seen substantial progress. Below we list the main timeline for this issue.
\begin{itemize}
    \item 1968: John E. Littlewood \cite{Littlewood} raised this question.

    \item 2005: Andr\'as Bezdek \cite{Bezdek} proved the first finite upper bound $\mathcal{N}(1,3)\leqslant24$.

    \item 2015: S. Boz\'oki, T.-L.Lee and L. R\'onyai \cite{BLR} constructed seven mutually touching infinite cylinders, showing $\mathcal{N}(1,3)\geqslant7$.

    \item 2025(June): Junnosuke Koizumi \cite{Koizumi} introduced a chirality graph and a signature \((3,3)\) linear-algebraic obstruction, proved that $\mathcal{N}(1,3)\leqslant18$.

    \item 2025(October): T. Dillon, J. Koizumi and S. Luo \cite{DKL} refined Koizumi's approach via additional forbidden graphs and Ramsey theory to obtain $\mathcal{N}(1,3)\leqslant10$.

    \item 2026(June): H\"ofer \cite{Hoefer} has shown that there are at most nine lines in $\R^3$ with pairwise distance one, so \(\mathcal{N}(1,3)\leqslant 9\).
\end{itemize}

But for general cases, we only know the following situations

\begin{itemize}
    \item $\mathcal{N}(0,n)=n+1$.

    \item $\mathcal{N}(n-1,n)=2$.

    \item 2025(December): Solymosi and Zahl \cite{SZ} proved that $\mathcal{N}(1,n)\leqslant4n\cdot 7^{2n-3}$ via methods from real algebraic geometry.
\end{itemize}

\textbf{To the best of our knowledge, prior to the present work, it was not known whether $\mathcal{N}(k,n)$ is finite for general $1<k<n-1$.}\\

The purpose of this paper is to prove that $\mathcal{N}(k,n)$ is finite for every pair $(k,n)$ with $0\leqslant k<n$.

For $k>1$, a direct extension of the line argument meets a new obstruction. If
\[
L_1=a_1+U_1,
\qquad
L_2=a_2+U_2,
\]
then the distance depends on the subspace $U_1+U_2$. In local coordinates, the usual Gram-matrix formula is rational only after choosing a nonvanishing minor, and the denominator degenerates whenever the rank of $U_1+U_2$ drops. For lines, the only such degeneration is parallelism. For higher-dimensional subspaces, several distinct ranks may occur.

The main observation of this paper is that this rank degeneracy can be handled uniformly. For a matrix $C$, let $\Delta_j(C)$ denote the sum of the squares of all $j\times j$ minors of $C$. If $A$ is a matrix whose columns span $U_1+U_2$, $w$ is a displacement vector between the two affine planes, and $\rank A=r$, then
\[
d(L_1,L_2)^2
=
\frac{\Delta_{r+1}([A,w])}{\Delta_r(A)}.
\]
This formula avoids choosing a particular nonzero Gram minor. For each fixed $r$ we therefore obtain a single polynomial relation $H_r(x,y)=0$ which contains the rank-$r$ unit-distance relation and also vanishes on the diagonal.

The second ingredient is the restricted Zariski closure of Solymosi and Zahl. If $X$ is a finite fixed-rank unit-distance configuration in one affine Grassmann chart, then $H_r$ vanishes on $X\times X$, hence it vanishes on the corresponding restricted closure $\overline X^{\,H_r}\times \overline X^{\,H_r}$. The variety $\overline X^{\,H_r}$ has uniformly bounded algebraic complexity. We refine $\overline X^{\,H_r}$ into a uniformly bounded number of connected Nash pieces. If one piece contained two points of $X$, then the rank-$r$ locus would be dense near the diagonal inside the square of that piece. This would produce pairs of affine planes arbitrarily close in the chart topology but still at distance one, which is impossible because the chart contains a distinguished point on each affine plane whose Euclidean distance tends to zero. Thus each piece contains at most one point of $X$.

Finally, in an arbitrary equidistant configuration, we color each pair by the integer
\[
\dim(U_i+U_j).
\]
There are only finitely many colors. Ramsey's theorem gives a large monochromatic subconfiguration. The color $r=k$ is the parallel case and is bounded by the ordinary equilateral-set bound in $\R^{n-k}$. Every color $r>k$ is bounded by the fixed-rank argument above.

To obtain an explicit upper bound for \(\mathcal{N}(k,n)\), we replace
the Ramsey step by a rank-descent argument.  For integer \(\rho\ 
(k\leqslant\rho\leqslant\min(2k,n-1))\), we consider configurations satisfying
\[
\dim(U_x+U_y)\leqslant\rho.
\]
The polynomial \(H_\rho\) vanishes not only on pairs of exact rank
\(\rho\), but also automatically on all pairs of smaller rank. After
passing to the restricted closure and decomposing it into connected
Nash pieces, the preceding diagonal argument shows that, inside each
piece, the maximal possible rank decreases from \(\rho\) to
\(\rho-1\). Iterating this descent until the parallel case is reached
gives a product bound involving the complexities of the relevant
Nash decompositions. Explicit estimates of Gabrielov-Vorobjov and
Basu-Pollack-Roy then yield our main result:
\begin{theorem}[Main theorem]
\label{thm:main}
For every pair of integers \(0\leqslant k<n\), one has
\[
\mathcal N(k,n)
\leqslant
2^{\,2^{\,10(k+1)(n-k)}}.
\]
In particular,
\[
\mathcal N(k,n)<\infty.
\]
\end{theorem}
The proof is effective in principle. It yields an upper bound that depends only on $k$ and $n$. We do not attempt to optimize the resulting numerical constant.

The paper is organized as follows. In Section~2 we introduce standard
coordinates on the affine Grassmannian. Section~3 establishes the
Gram--minor distance identity. In Section~4 we construct the
polynomials \(H_r\). Sections~5 and~6 develop the restricted Zariski
closure and the required Nash decomposition. Section~7 proves the
fixed-rank bound, and Section~8 treats the parallel case. Finally,
Section~9 gives both the Ramsey reduction and the rank-descent argument
leading to the explicit upper bound.

\section{The affine Grassmannian and standard charts}

Let $\Gr(k,n)$ denote the moduli space of $k$-dimensional linear subspace in $\R^n$.

\begin{proposition}\label{prop:affine-grassmann}
The affine Grassmannian $\Graff(k,n)$ is a smooth real algebraic manifold of dimension
\[
M=(k+1)(n-k).
\]
\end{proposition}

\begin{proof}
Every affine $k$-plane $L$ can be written uniquely in the form
\[
L=a+U,
\qquad
U\in\Gr(k,n),
\qquad
a\in U^\perp.
\]
Indeed, starting from any representation $L=a_0+U$, decompose
\[
a_0=a_U+a_\perp,
\qquad
a_U\in U,
\quad
a_\perp\in U^\perp.
\]
Then $L=a_\perp+U$, and uniqueness follows from $U\cap U^\perp=\{0\}$.

Thus $\Graff(k,n)$ is canonically the total space of the orthogonal-complement bundle
\[
\gamma_k^\perp
\longrightarrow
\Gr(k,n),
\]
whose fiber over $U$ is $U^\perp$. Since
\[
\dim\Gr(k,n)=k(n-k)
\]
and $\gamma_k^\perp$ has rank $n-k$, we obtain
\[
\dim\Graff(k,n)
=
k(n-k)+(n-k)
=
(k+1)(n-k).
\]
The Grassmannian and the orthogonal-complement bundle are algebraic, hence their total space is a smooth real algebraic manifold.
\end{proof}

We next recall the standard affine charts. Let
\[
I\subset\{1,\cdots,n\},
\qquad |I|=k.
\]
After permuting the coordinates, it is enough to consider
\[
I=\{1,\cdots,k\}.
\]
Write
\[
\R^n=\R^k\oplus\R^{n-k}.
\]
On the Grassmann chart consisting of $k$-planes transverse to $\{0\}\oplus\R^{n-k}$, every direction subspace is the graph of a unique linear map
\[
X:\R^k\to\R^{n-k}.
\]
Thus
\[
U_X
=
\left\{
\binom{u}{Xu}\bigg|u\in\R^k
\right\}.
\]
Every affine $k$-plane whose direction lies in this chart can be uniquely written as
\[
L_{X,b}
=
\left\{
\binom{u}{Xu+b}\bigg|u\in\R^k
\right\},
\qquad
X\in\Mat_{n-k,k}(\R),
\quad
b\in\R^{n-k}.
\]
Hence a standard affine Grassmann chart is identified with
\[
\R^M,
\qquad
M=(k+1)(n-k).
\]

There are exactly \(
\binom nk
\) standard Grassmann charts of this form, corresponding to the nonvanishing Pl\"ucker coordinates indexed by the $k$-subsets of $\{1,\cdots,n\}$. Their lifts cover $\Graff(k,n)$.

For later use, define
\[
B_X
:=
\begin{pmatrix}
I_k\\
X
\end{pmatrix}
\in\Mat_{n,k}(\R).
\]
Then
\[
U_X=\col(B_X).
\]
Moreover,
\[
p_{X,b}
:=
\binom{0}{b}
\in L_{X,b}.
\]
The point $p_{X,b}$ is not, in general, the orthogonal representative of the affine plane, but it depends polynomially on the chart coordinates and will be useful in the diagonal argument.

\section{Distance between affine subspaces}

We begin with the basic coordinate-free formula.

\begin{lemma}[Projection formula]\label{lem:projection}
Let
\[
L_1=a_1+U_1,
\qquad
L_2=a_2+U_2
\]
be affine subspaces of $\R^n$. Then
\[
d(L_1,L_2)
=
\dist(a_2-a_1,U_1+U_2)
=
\left\|\mathrm{Proj}_{(U_1+U_2)^\perp}(a_2-a_1)\right\|.
\]
\end{lemma}

\begin{proof}
For $u_i\in U_i$,
\[
(a_2+u_2)-(a_1+u_1)
=
(a_2-a_1)+(u_2-u_1),
\]
and
\[
\{u_2-u_1|u_1\in U_1,u_2\in U_2\}=U_1+U_2.
\]
Therefore
\[
d(L_1,L_2)
=
\inf_{v\in U_1+U_2}\|(a_2-a_1)-v\|,
\]
which is the Euclidean distance from $a_2-a_1$ to $U_1+U_2$. The final equality is the standard orthogonal-projection formula.
\end{proof}

\begin{corollary}\label{cor:positive-rank}
If $d(L_1,L_2)>0$, then
\[
U_1+U_2\ne\R^n.
\]
In particular,
\[
\dim(U_1+U_2)\leqslant n-1.
\]
\end{corollary}

\begin{proof}
If $U_1+U_2=\R^n$, then $(U_1+U_2)^\perp=\{0\}$, and Lemma \ref{lem:projection} gives $d(L_1,L_2)=0$.
\end{proof}

The ordinary distance function on $\Graff(k,n)\times\Graff(k,n)$ need not be continuous. This is one reason why the proof below is arranged differently from a naive connectedness argument.

\begin{example}\label{ex:discontinuous}
In $\R^2$, let
\[
L=\{(x,0)|x\in\R\}
\]
and
\[
L_\theta
=(0,1)+\R(\cos\theta,\sin\theta).
\]
For every $\theta\notin\mathbb{Z}\pi$, the lines intersect, hence $d(L,L_\theta)=0$. However,
\[
L_0=\{(x,1)|x\in\R\},
\qquad
d(L,L_0)=1.
\]
Thus $d$ is not continuous on the affine Grassmannian product $\Graff(1,2)\times\Graff(1,2)$.
\end{example}

\subsection{Sums of squares of minors}

Let $C\in\Mat_{p,q}(\mathbb{R})$. For $j\geqslant0$, define
\[
\Delta_j(C)
:=
\sum_{\substack{I\subset\{1,\cdots,p\},\ |I|=j\\
J\subset\{1,\cdots,q\},\ |J|=j}}
\det(C_{I,J})^2.
\]
We use the conventions
\[
\Delta_0(C)=1
\]
and
\[
\Delta_j(C)=0
\qquad
\text{if }j>\min(p,q).
\]

\begin{lemma}\label{lem:rank-delta}
For every real matrix $C$ and every $j\geqslant1$,
\[
\Delta_j(C)>0
\quad\Longleftrightarrow\quad
\rank C\geqslant j.
\]
Consequently,
\[
\rank C=r
\quad\Longleftrightarrow\quad
\Delta_r(C)>0
\text{ and }
\Delta_{r+1}(C)=0.
\]
\end{lemma}

\begin{proof}
The sum $\Delta_j(C)$ is a sum of squares. It is positive precisely when at least one $j\times j$ minor is nonzero, which is equivalent to $\rank C\geqslant j$.
\end{proof}

The next identity is the key algebraic input.

\begin{proposition}[Gram--minor distance identity]\label{prop:minor-distance}
Let $A\in\Mat_{n,m}(\R)$ have rank $r$, and let $w\in\R^n$. Then
\[\boxed{
\dist(w,\col A)^2
=
\frac{\Delta_{r+1}([A,w])}{\Delta_r(A)}.}
\]
Here $[A,w]$ denotes the matrix obtained by adjoining $w$ as a final column.
\end{proposition}

\begin{proof}
Choose an orthogonal matrix $Q\in O(n)$ sending $\col A$ to
\[
\R^r\times\{0\}\subset\R^r\oplus\R^{n-r}.
\]
Both sides of the desired identity are invariant under left multiplication of $A$ and $w$ by $Q$. For the quantities $\Delta_j$, this follows either from Cauchy--Binet or from the invariance of the singular values under orthogonal transformations. Thus we may assume
\[
A=
\begin{pmatrix}
A_0\\
0
\end{pmatrix},
\qquad
A_0\in\Mat_{r,m}(\R),
\qquad
\rank A_0=r.
\]
Write
\[
w=
\begin{pmatrix}
w_\parallel\\
w_\perp
\end{pmatrix},
\qquad
w_\parallel\in\R^r,
\quad
w_\perp\in\R^{n-r}.
\]
Then
\[
\dist(w,\col A)^2=\|w_\perp\|^2.
\]

Consider an $(r+1)\times(r+1)$ minor of $[A,w]$. Since $A$ has rank $r$, every nonzero such minor must contain the last column $w$. Moreover, because the last $n-r$ rows of $A$ are zero, a nonzero minor must use exactly one row among those last $n-r$ rows and all of the first $r$ rows. Therefore, if the extra row corresponds to the coordinate $(w_\perp)_\ell$ and $J$ is a set of $r$ columns of $A$, the determinant is, up to sign,
\[
(w_\perp)_\ell\det((A_0)_{[r],J}).
\]
Squaring and summing gives
\[
\Delta_{r+1}([A,w])
=
\|w_\perp\|^2
\sum_{|J|=r}\det((A_0)_{[r],J})^2.
\]
Since only the first $r$ rows of $A$ can contribute to an $r\times r$ minor,
\[
\Delta_r(A)
=
\sum_{|J|=r}\det((A_0)_{[r],J})^2.
\]
Hence
\[
\Delta_{r+1}([A,w])
=
\Delta_r(A)\|w_\perp\|^2,
\]
which proves the claim.
\end{proof}

\section{Polynomialization on a fixed rank stratum}

Fix one standard affine Grassmann chart. Write
\[
x=(X,b),
\qquad
y=(Y,c),
\]
with
\[
X,Y\in\Mat_{n-k,k}(\R),
\qquad
b,c\in\R^{n-k}.
\]
Define
\[
A(x,y)
:=
[B_X,B_Y]
=
\begin{pmatrix}
I_k&I_k\\
X&Y
\end{pmatrix}
\in\Mat_{n,2k}(\R),
\]
and
\[
w(x,y)
:=
\binom{0}{c-b}\in\R^n.
\]
Then
\[
\col A(x,y)=U_X+U_Y
\]
and, by Lemma \ref{lem:projection},
\[
d(L_x,L_y)
=
\dist(w(x,y),\col A(x,y)).
\]

Let
\[
r_*:=\min(2k,n-1).
\]
For pairwise positive-distance affine $k$-planes, Corollary \ref{cor:positive-rank} shows that the possible values of $\dim(U_X+U_Y)$ belong to
\[
\{k,k+1,\cdots,r_*\}.
\]

Fix an integer
\[
k<r\leqslant r_*.
\]
Define the polynomials
\[
F_r(x,y)
:=
\Delta_{r+1}([A(x,y),w(x,y)])
-
\Delta_r(A(x,y)),
\]
\[
R_r(x,y)
:=
\Delta_{r+1}(A(x,y)),
\]
and
\[
H_r(x,y)
:=
F_r(x,y)^2+R_r(x,y)^2.
\]

\begin{proposition}\label{prop:H-properties}
The polynomial $H_r$ has the following properties.
\begin{enumerate}[label=\textnormal{(\roman*)}]
\item If
\[
\rank A(x,y)=r
\quad\text{and}\quad
d(L_x,L_y)=1,
\]
then
\[
H_r(x,y)=0.
\]

\item If
\[
H_r(x,y)=0
\quad\text{and}\quad
\Delta_r(A(x,y))>0,
\]
then
\[
\rank A(x,y)=r
\quad\text{and}\quad
d(L_x,L_y)=1.
\]

\item For every $x$,
\[
H_r(x,x)=0.
\]

\item The total degree satisfies
\[
\deg H_r\leqslant 4(r+1).
\]
\end{enumerate}
\end{proposition}

\begin{proof}
Suppose first that $\rank A(x,y)=r$. By Proposition \ref{prop:minor-distance},
\[
d(L_x,L_y)^2
=
\frac{\Delta_{r+1}([A(x,y),w(x,y)])}{\Delta_r(A(x,y))}.
\]
Since $\rank A(x,y)=r$, we have
\[
\Delta_r(A(x,y))>0
\quad\text{and}\quad
\Delta_{r+1}(A(x,y))=0.
\]
If $d(L_x,L_y)=1$, then
\[
\Delta_{r+1}([A,w])=\Delta_r(A),
\]
so $F_r=R_r=0$, proving (i).

Now assume $H_r(x,y)=0$ and $\Delta_r(A(x,y))>0$. Since $H_r$ is a sum of two squares,
\[
F_r(x,y)=0,
\qquad
R_r(x,y)=0.
\]
The second equality gives
\[
\Delta_{r+1}(A(x,y))=0,
\]
while the assumed positivity of $\Delta_r$ gives
\[
\rank A(x,y)\geqslant r.
\]
By Lemma \ref{lem:rank-delta},
\[
\rank A(x,y)=r.
\]
The equality $F_r=0$, together with Proposition \ref{prop:minor-distance}, now yields
\[
d(L_x,L_y)^2=1.
\]
This proves (ii).

For (iii), note that
\[
A(x,x)=[B_X,B_X]
\]
has rank $k<r$, while
\[
w(x,x)=0.
\]
Therefore
\[
\Delta_r(A(x,x))=0,
\quad
\Delta_{r+1}(A(x,x))=0,
\quad
\Delta_{r+1}([A(x,x),w(x,x)])=0,
\]
and hence $H_r(x,x)=0$.

Finally, every entry of $A(x,y)$ and $w(x,y)$ is affine linear in the chart variables. A $j\times j$ minor therefore has degree at most $j$, and its square has degree at most $2j$. Thus
\[
\deg F_r\leqslant 2(r+1),
\qquad
\deg R_r\leqslant 2(r+1),
\]
and hence
\[
\deg H_r\leqslant 4(r+1).
\]
\end{proof}

\begin{remark}
The point of $H_r$ is that it simultaneously remembers the unit-distance equation and the upper-rank condition
\[
\rank A(x,y)\leqslant r.
\]
The additional hypothesis $\Delta_r(A(x,y))>0$ then recovers the exact rank $r$. No choice of a nonzero Gram minor is required.
\end{remark}

\section{Restricted Zariski closure}

We use the symmetric restricted Zariski closure introduced by Solymosi and Zahl \cite{SZ}.

Let
\[
H:\R^M\times\R^M\to\R
\]
be a real polynomial. For $a,b\in\R^M$, define the slices
\[
H_a(y):=H(a,y),
\qquad
H^b(x):=H(x,b).
\]

\begin{definition}[Restricted Zariski closure]\label{def:rzc}
For $X\subset\R^M$, define
\[
\overline X^{\,H}
:=
\bigcap_{\substack{a\in\R^M\\X\subset Z(H_a)}}Z(H_a)
\cap
\bigcap_{\substack{b\in\R^M\\X\subset Z(H^b)}}Z(H^b).
\]
If one of the indexing families is empty, the corresponding intersection is interpreted as $\R^M$.
\end{definition}

\begin{lemma}[Propagation lemma]\label{lem:propagation}
Let $X\subset\R^M$ be nonempty and suppose
\[
H(x,y)=0
\qquad
(x,y\in X).
\]
Then
\[
H(x,y)=0
\qquad
(x,y\in\overline X^{\,H}).
\]
\end{lemma}

\begin{proof}
See \cite[Lemma 3.1]{SZ}.
\end{proof}

For our purposes, we also need a uniform bound on the algebraic complexity of $\overline X^{\,H}$. This is elementary because all defining slices lie in a finite-dimensional polynomial space.

\begin{lemma}[Finite defining family]\label{lem:finite-defining}
Suppose $H$ has degree at most $D$ in each block of variables. Let
\[
S(M,D):=\binom{M+D}{D}.
\]
Then, for every $X\subset\R^M$, the restricted closure $\overline X^{\,H}$ can be defined by at most $S(M,D)$ polynomial equations, each of degree at most $D$.
\end{lemma}

\begin{proof}
Let $\mathcal P_{M,D}$ denote the vector space of real polynomials in $M$ variables of degree at most $D$. Its dimension is
\[
\dim\mathcal P_{M,D}=S(M,D).
\]
Consider the family $\mathcal F_X\subset\mathcal P_{M,D}$ consisting of all slices $H_a$ and $H^b$ whose zero sets contain $X$. By definition,
\[
\overline X^{\,H}
=
\bigcap_{P\in\mathcal F_X}Z(P).
\]
Let
\[
W_X:=\Span_{\R}(\mathcal F_X)\subset\mathcal P_{M,D}.
\]
Choose a basis
\[
P_1,\cdots,P_s
\]
of $W_X$. Then
\[
s\leqslant S(M,D).
\]
A point $z$ lies in the zero set of every polynomial in $\mathcal F_X$ if and only if it lies in the zero set of every polynomial in $W_X$, which is equivalent to
\[
P_1(z)=\cdots=P_s(z)=0.
\]
Hence
\[
\overline X^{\,H}=Z(P_1,\cdots,P_s).
\]
\end{proof}

\section{Uniform Nash decomposition}

We use a standard consequence of effective semialgebraic geometry. We state it in the form needed here.

\begin{definition}
A subset \(S\subset\R^M\) is called a \emph{Nash submanifold}
of \(\R^M\) if \(S\) is both a semialgebraic subset of \(\R^M\)
and a real-analytic embedded submanifold of \(\R^M\).
It is called \emph{connected} if it is connected in the Euclidean
topology.
\end{definition}

\begin{theorem}[Uniform Nash decomposition]
\label{thm:nash-stratification}
For every triple of positive integers \(M,s,D\), there exists a positive integer
\[
\Sigma(M,s,D)
\]
with the following property. If
\[
Z\subset\R^M
\]
is a real algebraic set defined by at most \(s\) polynomial equations
of degree at most \(D\), then \(Z\) admits a partition
\[
Z=S_1\sqcup\cdots\sqcup S_N
\]
with
\[
N\leqslant\Sigma(M,s,D),
\]
where every \(S_j\) is a connected Nash submanifold of \(\R^M\).
\end{theorem}

\begin{proof}
Write
\[
Z=V(f_1,\cdots,f_t),
\qquad
t\leqslant s,
\qquad
\deg f_i\leqslant D.
\]
Regarded as a semialgebraic set, \(Z\) is defined by a formula of
format depending only on \(M,s,D\).

By the dense semialgebraic stratification theorem of Gabrielov and
Vorobjov
\cite[Section~6, Corollary~2]{GabrielovVorobjov},
there exist constants
\[
L_0=L_0(M,s,D),\qquad
s_0=s_0(M,s,D),\qquad
D_0=D_0(M,s,D)
\]
and a finite disjoint decomposition
\[
Z=T_1\sqcup\cdots\sqcup T_L,
\qquad
L\leqslant L_0,
\]
such that each \(T_\alpha\) is a smooth semialgebraic submanifold
and admits a description by at most \(s_0\) polynomial equations
and strict inequalities, all of degree at most \(D_0\).

More precisely, after relabeling the defining polynomials, each
\(T_\alpha\) has the form
\[
T_\alpha
=
\left\{
x\in\R^M\big|
q_{\alpha,1}(x)=\cdots=q_{\alpha,a_\alpha}(x)=0,\;
p_{\alpha,1}(x)>0,\cdots,p_{\alpha,b_\alpha}(x)>0
\right\},
\]
where
\[
a_\alpha+b_\alpha\leqslant s_0
\]
and all the polynomials involved have degree at most \(D_0\). By the definition of the elementary strata in Gabrielov-Vorobjov \cite{GabrielovVorobjov}, each \(T_\alpha\) is locally defined by polynomial equations whose Jacobian has maximal rank. Hence the real analytic implicit function theorem shows that \(T_\alpha\) is a real analytic submanifold. Since it is also semialgebraic, it is a Nash submanifold.

It remains to control the number of connected components of the
\(T_\alpha\).  Put
\[
Q_\alpha
=
\{q_{\alpha,1},\cdots,q_{\alpha,a_\alpha}\},
\qquad
P_\alpha
=
\{p_{\alpha,1},\cdots,p_{\alpha,b_\alpha}\}.
\]
Then \(T_\alpha\) is one realization of a sign condition of
\(P_\alpha\) on the algebraic set
\[
V(Q_\alpha).
\]
By
\cite[Theorem~1.1]{BPRsign}, the sum of the zeroth Betti numbers
of all realizable sign conditions of \(P_\alpha\) on
\(V(Q_\alpha)\) is bounded by a constant depending only on
\(M,s_0,D_0\).  Consequently, there exists
\[
C_0=C_0(M,s_0,D_0)
\]
such that
\[
b_0(T_\alpha)\leqslant C_0
\]
for every \(\alpha\).

The connected components of a semialgebraic set are semialgebraic.
Since each connected component of \(T_\alpha\) is an open and closed
submanifold of \(T_\alpha\), it is again a Nash submanifold.
Splitting every \(T_\alpha\) into its connected components therefore
gives a partition
\[
Z=S_1\sqcup\cdots\sqcup S_N
\]
into connected Nash submanifolds, with
\[
N
\leqslant
L_0(M,s,D)\,
C_0\bigl(M,s_0(M,s,D),D_0(M,s,D)\bigr).
\]
Thus the conclusion follows by defining the right-hand side to be
\(\Sigma(M,s,D)\).
\end{proof}

\begin{remark}
Only the existence of a uniform bound depending on \(M,s,D\) is
needed in the qualitative argument. No explicit estimate for
\(\Sigma(M,s,D)\) is required there.
\end{remark}

We shall also use the elementary identity principle for real-analytic functions.

\begin{lemma}[Analytic zero sets] \label{lem:analytic-zero} 
Let \(S\) be a connected real-analytic manifold, and let \[ f:S\longrightarrow\R \] be real analytic. If \(f\not\equiv0\), then its zero set \[ Z(f):=\{x\in S|f(x)=0\} \] has empty interior in \(S\). Equivalently, the nonzero locus \[ \{x\in S|f(x)\ne0\} \] is dense in \(S\). 
\end{lemma}

\begin{proof}
This is the usual identity theorem for real-analytic functions on connected real-analytic manifolds.
\end{proof}

\section{A uniform bound on fixed-rank configurations}

We now prove the main algebraic-geometric step.

\begin{theorem}[Fixed-rank bound]\label{thm:fixed-rank}
Fix integers $0<k<n$ and
\[
k<r\leqslant\min(2k,n-1).
\]
There exists a finite constant
\[
B(k,n,r)
\]
with the following property. Let $X$ be a finite set of points in one standard affine Grassmann chart such that
\[
d(L_x,L_y)=1
\]
and
\[
\dim(U_x+U_y)=r
\]
for every two distinct points $x,y\in X$. Then
\[
|X|\leqslant B(k,n,r).
\]
\end{theorem}

\begin{proof}
Let
\[
M=(k+1)(n-k),
\qquad
D_r:=4(r+1),
\]
and let $H_r$ be the polynomial from Proposition \ref{prop:H-properties}. By hypothesis and part (i) of that proposition,
\[
H_r(x,y)=0
\qquad
(x\ne y,\ x,y\in X).
\]
Part (iii) gives
\[
H_r(x,x)=0
\qquad
(x\in X).
\]
Therefore
\[
H_r(x,y)=0
\qquad
(x,y\in X).
\]
By Lemma \ref{lem:propagation},
\[
H_r(u,v)=0
\qquad
(u,v\in \overline{X}^{\,H_r}).
\]
By Lemma \ref{lem:finite-defining}, the set $\overline{X}^{\,H_r}$ can be defined by at most
\[
s_r:=S(M,D_r)=\binom{M+D_r}{D_r}
\]
polynomials of degree at most $D_r$. Apply Theorem \ref{thm:nash-stratification}. We obtain a partition
\[
\overline{X}^{\,H_r}=S_1\sqcup\cdots\sqcup S_N,
\qquad
N\leqslant\Sigma(M,s_r,D_r),
\]
where each $S_j$ is a connected Nash submanifold.

We claim that
\[
|X\cap S_j|\leqslant1
\]
for every $j$. Suppose, to the contrary, that one piece $S:=S_j$ contains two distinct points
\[
x,y\in X.
\]
Consider the polynomial function
\[
G(u,v)
:=
\Delta_r(A(u,v))
\]
restricted to the connected Nash manifold
\[
S\times S.
\]
Since the pair $(x,y)$ has rank $r$,
\[
G(x,y)>0.
\]
Thus $G$ is not identically zero on $S\times S$.

Fix one point $z\in S$; for instance, take $z=x$. The product $S\times S$ is connected, and $G$ is real analytic on it. By Lemma \ref{lem:analytic-zero}, every neighborhood of $(z,z)$ in $S\times S$ contains a point $(u,v)$ with
\[
G(u,v)\ne0.
\]
Since $G$ is a sum of squares, this means
\[
G(u,v)>0.
\]
Hence we may choose a sequence
\[
(u_m,v_m)\in S\times S
\]
such that
\[
(u_m,v_m)\to(z,z)
\]
and
\[
\Delta_r(A(u_m,v_m))>0
\]
for every $m$.

Because $u_m,v_m\in \overline{X}^{\,H_r}$, we have
\[
H_r(u_m,v_m)=0.
\]
By part (ii) of Proposition \ref{prop:H-properties}, it follows that
\[
\rank A(u_m,v_m)=r
\]
and
\[
d(L_{u_m},L_{v_m})=1
\]
for every $m$.

Write
\[
u_m=(X_m,b_m),
\qquad
v_m=(Y_m,c_m).
\]
Since
\[
(u_m,v_m)\to(z,z),
\]
we have
\[
\|b_m-c_m\|\to0.
\]
On the other hand,
\[
p_{u_m}:=\binom{0}{b_m}\in L_{u_m},
\qquad
p_{v_m}:=\binom{0}{c_m}\in L_{v_m}.
\]
Therefore
\[
1=d(L_{u_m},L_{v_m})
\leqslant
\|p_{u_m}-p_{v_m}\|
=
\|b_m-c_m\|\longrightarrow0,
\]
which is impossible.

Thus each Nash piece contains at most one point of $X$. Consequently,
\[
|X|
\leqslant
N
\leqslant
\Sigma(M,s_r,D_r).
\]
We may therefore take
\[
B(k,n,r)
:=
\Sigma\!\left(
(k+1)(n-k),
\binom{(k+1)(n-k)+4(r+1)}{4(r+1)},
4(r+1)
\right).
\]
\end{proof}

\begin{remark}\label{rem:diagonal}
The proof does not use continuity of the affine-subspace distance function. Indeed, that function is not continuous, as shown in Example \ref{ex:discontinuous}. The contradiction uses only the elementary estimate
\[
d(L_{X,b},L_{Y,c})\leqslant\|b-c\|
\]
for pairs converging to the diagonal in one affine Grassmann chart.
\end{remark}

\section{The parallel case}

Before applying Ramsey theory, we isolate the rank-$k$ case.

\begin{proposition}[Parallel bound]\label{prop:parallel}
Let
\[
\mathcal L=\{L_1,\cdots,L_m\}
\]
be affine $k$-planes in $\R^n$ with a common direction $U$, and suppose
\[
d(L_i,L_j)=1
\qquad(i\ne j).
\]
Then
\[
m\leqslant n-k+1.
\]
\end{proposition}

\begin{proof}
Write each affine plane in its orthogonal form
\[
L_i=a_i+U,
\qquad
a_i\in U^\perp.
\]
By Lemma \ref{lem:projection},
\[
d(L_i,L_j)=\|a_i-a_j\|.
\]
Thus $a_1,\cdots,a_m$ form an equilateral set in the Euclidean space
\[
U^\perp\cong\R^{n-k}.
\]
Translate so that $a_m=0$ and put
\[
v_i=a_i-a_m
\qquad(1\leqslant i<m).
\]
Then
\[
\|v_i\|^2=1
\]
and, for $i\ne j$,
\[
1=\|v_i-v_j\|^2
=2-2v_i\cdot v_j,
\]
so
\[
v_i\cdot v_j=\frac12.
\]
The Gram matrix of $v_1,\cdots,v_{m-1}$ is
\[
\frac12(I_{m-1}+J_{m-1}),
\]
which is positive definite. Hence the vectors are linearly independent and
\[
m-1\leqslant n-k.
\]
\end{proof}

\section{Ramsey reduction, rank descent and explicit bounds}

We now combine the fixed-rank theorem with finite Ramsey theory.

For integers $q,t\geqslant1$, let
\[
R_q(t)
\]
denote the diagonal $q$-color Ramsey number: every edge-coloring of the complete graph on $R_q(t)$ vertices with $q$ colors contains a monochromatic complete subgraph on $t$ vertices.

\begin{theorem}[Quantitative finite bound]\label{thm:quantitative}
Let $0<k<n$, put
\[
r_*:=\min(2k,n-1),
\qquad
q:=r_*-k+1,
\]
and define
\[
B_*:=
\max_{k<r\leqslant r_*}B(k,n,r),
\]
with the convention $B_*=0$ if the index set is empty. Let
\[
T:=1+\max\{n-k+1,B_*\}.
\]
Then
\[\boxed{
\mathcal{N}(k,n)
\leqslant
\binom nk\bigl(R_q(T)-1\bigr).}
\]
In particular,
\[
\mathcal{N}(k,n)<\infty.
\]
\end{theorem}

\begin{proof}
Let
\[
\mathcal L=\{L_1,\cdots,L_m\}
\]
be a family of affine $k$-planes satisfying
\[
d(L_i,L_j)=1
\qquad(i\ne j).
\]
Assign each $L_i$ to one of the standard affine Grassmann charts containing it. There are \(\binom nk\) such charts. If
\[
m>\binom nk\cdot\bigl(R_q(T)-1\bigr),
\]
then by the pigeonhole principle one standard chart contains at least
\[
R_q(T)
\]
members of the family. Restrict to those members.

For each pair $i\ne j$, let $U_i,U_j$ denote the direction subspaces and color the edge $\{i,j\}$ by
\[
r_{ij}:=\dim(U_i+U_j).
\]
Since $\dim U_i=\dim U_j=k$,
\[
r_{ij}\geqslant k.
\]
Since $d(L_i,L_j)=1>0$, Corollary \ref{cor:positive-rank} gives
\[
r_{ij}\leqslant n-1,
\]
also \(r_{ij}\leqslant2k\). Hence
\[
r_{ij}\in\{k,k+1,\cdots,r_*\},
\]
so there are exactly $q$ possible colors.

By the definition of $R_q(T)$, there is a subfamily
\[
\mathcal L'=\{L_{i_1},\cdots,L_{i_T}\}
\]
of size $T$ for which
\[
\dim(U_{i_a}+U_{i_b})=r
\]
is constant for all $a\ne b$.

If $r=k$, then
\[
\dim(U_{i_a}+U_{i_b})=k
\]
forces
\[
U_{i_a}=U_{i_b}
\]
for every pair. Thus all members of $\mathcal L'$ are parallel. By Proposition \ref{prop:parallel},
\[
T\leqslant n-k+1,
\]
contradicting the definition of $T$.

If $r>k$, then all members lie in the same affine Grassmann chart and satisfy the hypotheses of Theorem \ref{thm:fixed-rank}. Therefore
\[
T\leqslant B(k,n,r)\leqslant B_*,
\]
again contradicting the definition of $T$.

Thus
\[
m\leqslant\binom nk\bigl(R_q(T)-1\bigr),
\]
which proves the theorem.
\end{proof}

The preceding theorem already establishes the finiteness of \(\mathcal N(k,n)\) by combining the fixed-rank bound with finite Ramsey theory. However, the resulting estimate still involves the abstract constants \(B(k,n,r)\) and is therefore not fully explicit. We now refine the argument by replacing the Ramsey reduction with a rank-descent procedure. Using the same polynomial relations \(H_\rho\), together with explicit complexity bounds for the associated Nash decompositions, we obtain a completely explicit double-exponential upper bound.

\begin{lemma}[Rank-descent recursion]
\label{lem:recursion}
Fix integers
\[
1\leqslant k\leqslant n-2,
\]
and put
\[
r_*:=\min(2k,n-1),
\qquad
M:=(k+1)(n-k).
\]
For every integer
\(\nu\ (
k\leqslant \nu\leqslant r_*)
\),
let \(A_\nu=A_\nu(k,n)\) denote the supremum of the
cardinalities of finite families \(X\) contained in one fixed
standard affine Grassmann chart and satisfying
\[
d(L_x,L_y)=1,
\qquad
\dim(U_x+U_y)\leqslant \nu
\]
for all distinct \(x,y\in X\). Fix an integer
\(\rho\ (
k<\rho\leqslant r_*)
\).
Set
\[
D_\rho:=4(\rho+1),
\qquad
\beta_\rho:=2D_\rho=8(\rho+1),
\]
and
\[
E_\rho
:=
\left[
\beta_\rho(\beta_\rho+1)
\right]^{24^M},\qquad
L_\rho
:=
2^{M+1}E_\rho^{M^2},
\qquad
\delta_\rho
:=
2^M(\beta_\rho+1)E_\rho^{M-2}.
\]
Finally, define
\[
\Sigma_\rho
:=
L_\rho\,
\delta_\rho(2\delta_\rho-1)^{M-1}
\sum_{j=0}^{M}
\binom{2^M}{j}4^j.
\]
Then
\begin{equation}
\label{eq:rank-descent-recursion}
\boxed{A_\rho
\leqslant
\Sigma_\rho A_{\rho-1}.}
\end{equation}
\end{lemma}

\begin{proof}
    Let \(X\) be a finite family contained in the fixed standard
affine Grassmann chart and satisfying
\[
d(L_x,L_y)=1,
\qquad
\dim(U_x+U_y)\leqslant\rho
\]
for all distinct \(x,y\in X\). If \(X=\varnothing\), there is nothing to prove. Hence assume \(X\ne\varnothing\). We first show that
\[
H_\rho(x,y)=0
\qquad
(x,y\in X).
\]
If
\[
\dim(U_x+U_y)=\rho,
\]
then this follows from
Proposition~\ref{prop:H-properties}\textnormal{(i)}. Suppose instead that
\[
\dim(U_x+U_y)\leqslant\rho-1.
\]
Then
\[
\rank A(x,y)\leqslant\rho-1,
\]
and hence
\[
\Delta_\rho(A(x,y))=0,\qquad
\Delta_{\rho+1}(A(x,y))=0.
\]
Moreover, adjoining one column can increase the rank by at most one,
so
\[
\rank[A(x,y),w(x,y)]
\leqslant
\rank A(x,y)+1
\leqslant\rho.
\]
It follows that
\[
\Delta_{\rho+1}([A(x,y),w(x,y)])=0.
\]
Consequently,
\[
F_\rho(x,y)=0,
\qquad
R_\rho(x,y)=0,
\]
and therefore
\[
H_\rho(x,y)=0.
\]
The diagonal identity
\(
H_\rho(x,x)=0
\)
follows from
Proposition~\ref{prop:H-properties}\textnormal{(iii)}. Thus
\[
H_\rho(x,y)=0
\qquad
(x,y\in X).
\]
By Lemma~\ref{lem:propagation},
\[
H_\rho(u,v)=0
\qquad
(u,v\in \overline X^{\,H_\rho}).
\]
By Lemma~\ref{lem:finite-defining}, there exist polynomials
\(
P_1,\cdots,P_s
\)
of degree at most \(D_\rho\) such that
\[
\overline X^{\,H_\rho}=V(P_1,\cdots,P_s).
\]

If \(s=0\), then
\[
\overline X^{\,H_\rho}=\R^M.
\]
In this case \(\overline X^{\,H_\rho}\) itself is a connected Nash submanifold. The
diagonal argument given below, applied with \(S=\overline X^{\,H_\rho}\), shows that no
two points \(x,y\in X\) can satisfy
\[
\dim(U_x+U_y)=\rho.
\]
Hence all pairs in \(X\) have direction-sum dimension at most
\(\rho-1\), so
\[
|X|\leqslant A_{\rho-1}
\leqslant
\Sigma_\rho A_{\rho-1}.
\]
We may therefore assume
\(
s\geqslant1
\). Set
\[
P:=P_1^2+\cdots+P_s^2.
\]
Since the ground field is \(\R\), we have
\[
\overline X^{\,H_\rho}=V(P),\qquad
\deg P
\leqslant
2D_\rho
=
\beta_\rho.
\]

We now apply the explicit stratification procedure of Gabrielov-Vorobjov to the elementary semialgebraic set
\[
\overline X^{\,H_\rho}=\{P=0\}.
\]
In their notation, we may take
\[
I=1,
\qquad
J=0,
\qquad
r=0,
\qquad
\alpha=1,
\qquad
\beta=\beta_\rho.
\]
By
\cite[Lemma~5, Theorem~2, and Remark~2]{GabrielovVorobjov},
all auxiliary degree parameters and derivative-order parameters
produced by the procedure are bounded by
\[
\left[
\beta_\rho(\beta_\rho+1)
\right]^{24^M}
=
E_\rho.
\]

The counting estimate in the proof of
\cite[Theorem~2]{GabrielovVorobjov} bounds the number of strata
arising from the nonzero sequences by
\[
2^M E_\rho^{M^2}.
\]
Allowing also for the additional open stratum occurring in the
algorithm, the total number of smooth semialgebraic strata is
therefore bounded by
\[
L_\rho
=
2^{M+1}E_\rho^{M^2}.
\]

We next bound the ordinary degrees of all dense polynomials appearing
in the descriptions of these strata. Since the input polynomial has
degree at most \(\beta_\rho\), we use
\[
d:=\beta_\rho+1
\]
in order to conform to the convention in
\cite{GabrielovVorobjov} that the input degrees are strictly smaller
than \(d\). By
\cite[Remark~4]{GabrielovVorobjov}, the degrees of all output
polynomials are strictly smaller than
\[
2^{M-1}d\,M_M^{M-2},
\]
where \(M_M\) denotes the final derivative-order parameter in the
algorithm. Since
\(
M_M\leqslant E_\rho,
\)
every output polynomial has degree at most
\[
\delta_\rho
=
2^M(\beta_\rho+1)E_\rho^{M-2}.
\]

Furthermore, Theorem~2 of Gabrielov-Vorobjov \cite{GabrielovVorobjov} shows that, since
\(J=0\), every stratum is described using at most \(2^M\) strict polynomial inequalities.

Let \(T\) be one of these smooth strata. Let
\(
\mathcal Q_T
\)
denote the family of polynomial equations occurring in a description
of \(T\), and let
\(
\mathcal P_T
\)
denote the family of polynomials occurring in its strict
inequalities. Then
\[
\#\mathcal P_T\leqslant2^M
\]
and
\[
\deg Q,\deg R\leqslant\delta_\rho,
\qquad \text{for every}\ Q\in\mathcal Q_T,
R\in\mathcal P_T.
\]
The stratum \(T\) is a realization of a sign condition of
\(\mathcal P_T\) on the real algebraic set
\(
V(\mathcal Q_T).
\)

Let \(k_T:=\dim_{\R}V(\mathcal Q_T)\). then \(k_T\leqslant M\). By
\cite[Theorem~1.1]{BPRsign}, applied with \(i=0\),
\[
\sum_{\sigma}
b_0\bigl(\mathcal R(\sigma,V(\mathcal Q_T))\bigr)
\leqslant
\delta_\rho(2\delta_\rho-1)^{M-1}
\sum_{j=0}^{k_T}
\binom{\#\mathcal P_T}{j}4^j.
\]
Since \(k_T\leqslant M, \#\mathcal P_T\leqslant2^M\), the right-hand side is at most
\[
\delta_\rho(2\delta_\rho-1)^{M-1}
\sum_{j=0}^{M}
\binom{2^M}{j}4^j.
\]
In particular, \(T\) has at most this many connected components.

For semialgebraic sets, the semialgebraically connected components
coincide with the ordinary connected components. Moreover, every
connected component of a smooth semialgebraic stratum is a connected
Nash submanifold. It follows that \(\overline X^{\,H_\rho}\) admits a partition into at
most
\[
\Sigma_\rho
=
L_\rho\,
\delta_\rho(2\delta_\rho-1)^{M-1}
\sum_{j=0}^{M}
\binom{2^M}{j}4^j
\]
connected Nash submanifolds.

Let \(S\) be one of these connected Nash pieces. We claim that
\[
\dim(U_x+U_y)\leqslant\rho-1
\]
for all distinct
\(
x,y\in X\cap S.
\)
Suppose, to the contrary, that
\(
x,y\in X\cap S
\)
and
\[
\dim(U_x+U_y)=\rho.
\]
Consider the polynomial function
\[
G(u,v):=\Delta_\rho(A(u,v))
\]
on the connected real-analytic manifold
\(
S\times S
\). Since
\[
G(x,y)>0,
\]
the function \(G\) is not identically zero on \(S\times S\). By
Lemma~\ref{lem:analytic-zero}, its nonzero locus is dense in
\(S\times S\). Thus, for any \(z\in S\), there exists a sequence
\[
(u_m,v_m)\longrightarrow(z,z)
\]
in \(S\times S\) such that
\[
\Delta_\rho(A(u_m,v_m))>0
\]
for every \(m\). Since \(u_m,v_m\in \overline X^{\,H_\rho}\), we have
\[
H_\rho(u_m,v_m)=0.
\]
By
Proposition~\ref{prop:H-properties}\textnormal{(ii)}, it follows
that
\[
\rank A(u_m,v_m)=\rho,\qquad
d(L_{u_m},L_{v_m})=1
\]
for every \(m\). Writing
\[
u_m=(X_m,b_m),
\qquad
v_m=(Y_m,c_m),
\]
we have
\[
\|b_m-c_m\|\longrightarrow0.
\]
Since
\[
\binom{0}{b_m}\in L_{u_m},\qquad
\binom{0}{c_m}\in L_{v_m},
\]
it follows that
\[
1
=
d(L_{u_m},L_{v_m})
\leqslant
\|b_m-c_m\|
\longrightarrow0,
\]
which is impossible.

Hence every set \(X\cap S\) satisfies
\[
\dim(U_x+U_y)\leqslant\rho-1
\]
for all distinct \(x,y\in X\cap S\). By the definition of
\(A_{\rho-1}\),
\[
|X\cap S|
\leqslant
A_{\rho-1}.
\]
Since there are at most \(\Sigma_\rho\) connected Nash pieces,
\[
|X|
\leqslant
\Sigma_\rho A_{\rho-1}.
\]
Taking the supremum over all such \(X\) proves
\eqref{eq:rank-descent-recursion}.
\end{proof}

\begin{theorem}[A fully explicit upper bound]
\label{thm:fully-explicit-upper-bound}
For all integers \(0\leqslant k<n\), one has
\[\boxed{
\mathcal N(k,n)
\leqslant
2^{\,2^{\,10(k+1)(n-k)}}.}
\]
In particular, the upper bound is completely explicit and depends
only on \(k\) and \(n\).
\end{theorem}

\begin{proof}
The cases \(k=0\) and \(k=n-1\) are classical. It therefore remains to consider
\[
1\leqslant k\leqslant n-2.
\]

Put
\[
r_*:=\min(2k,n-1),\qquad
M:=(k+1)(n-k).
\]
Since
\(
M-n
=
(k+1)(n-k)-n
=
k(n-k-1)
\geqslant0
\),
we have \(M\geqslant n\). Moreover, \(M\geqslant4\).

All standard affine Grassmann charts are carried to one another by
permutations of the ambient coordinates. Since coordinate
permutations are orthogonal transformations, they preserve both the
separation distance and the dimensions of sums of direction
subspaces. Consequently, all the chartwise bounds obtained below are
uniform over the standard charts.

For integer \(\rho\ (k\leqslant\rho\leqslant r_*)\), let \(A_\rho=A_\rho(k,n)\) be as in
Lemma~\ref{lem:recursion}. By
Proposition~\ref{prop:parallel},
\[
A_k\leqslant n-k+1.
\]
Moreover, for every integer \(\rho\ (k<\rho\leqslant r_*)\), Lemma~\ref{lem:recursion} gives
\[
A_\rho
\leqslant
\Sigma_\rho A_{\rho-1}.
\]
Iterating the recursion from \(\rho=k+1\) to \(\rho=r_*\), we obtain
\[
A_{r_*}
\leqslant
(n-k+1)
\prod_{\rho=k+1}^{r_*}
\Sigma_\rho.
\]
There are \(\binom nk\) standard affine Grassmann charts. Assigning
each member of a finite unit-distance family to one standard chart
containing it therefore gives
\begin{equation}
\label{eq:explicit-product-bound}
\mathcal N(k,n)
\leqslant
\binom nk
(n-k+1)
\prod_{\rho=k+1}^{r_*}
\Sigma_\rho.
\end{equation}

It remains to simplify this explicit product. Since
\(
\rho+1
\leqslant
r_*+1
\leqslant
n
\leqslant
M
\),
we have
\[
\beta_\rho
=
8(\rho+1)
\leqslant
8M
\]
and consequently
\[
\beta_\rho(\beta_\rho+1)
\leqslant
8M(8M+1)
\leqslant
72M^2.
\]
For \(M\geqslant4\),
\(
\log_2(72M^2)\leqslant2^M
\). Hence
\[
\log_2E_\rho
\leqslant
24^M\log_2(72M^2)
\leqslant
24^M2^M
=
48^M
\leqslant
2^{6M}\Longrightarrow
E_\rho
\leqslant
2^{\,2^{6M}}.
\]

Next,
\begin{align*}
\log_2\delta_\rho
=
M+\log_2(\beta_\rho+1)
+(M-2)\log_2E_\rho\leqslant
3M+M2^{6M}.
\end{align*}
Also,
\[
\sum_{j=0}^{M}
\binom{2^M}{j}4^j
\leqslant
\sum_{j=0}^{2^M}
\binom{2^M}{j}4^j
=
5^{2^M}.
\]
It follows that, for $M\geqslant4$
\begin{align*}
\log_2\Sigma_\rho
&\leqslant
(M+1)
+
M^2\log_2E_\rho
+
\log_2\delta_\rho
+
(M-1)\log_2(2\delta_\rho)
+
2^M\log_2 5\\
&\leqslant
2M^2\,2^{6M}
+
3M^2
+
2M
+
3\cdot2^M\\
&\leqslant3M^22^{6M}\leqslant2^{8M}.
\end{align*}
Hence
\[
\Sigma_\rho
\leqslant
2^{\,2^{8M}}.
\]

Finally, the number of factors in
\eqref{eq:explicit-product-bound} is
\(
r_*-k\leqslant M
\).
Moreover,
\[
\binom nk
\leqslant
2^n
\leqslant
2^M,\qquad
n-k+1
\leqslant
n+1
\leqslant
2^n
\leqslant
2^M.
\]
Consequently,
\[
\binom nk(n-k+1)
\leqslant
2^{2M}.
\]
It follows that, for $M\geqslant4$
\[
\log_2\mathcal N(k,n)
\leqslant
2M+M2^{8M}
\leqslant
2^{10M}\Longrightarrow
\mathcal N(k,n)
\leqslant
2^{\,2^{10M}}
=
2^{\,2^{\,10(k+1)(n-k)}}.
\]
This completes the proof.
\end{proof}

\begin{proof}[Proof of Theorem \ref{thm:main}]
It suffices to consider finite families. Indeed, an infinite pairwise unit-distance family would contain finite subfamilies of arbitrarily large cardinality. For $0\leqslant k<n$, apply Theorem \ref{thm:fully-explicit-upper-bound}.
\end{proof}

\section{Further remarks and open problems}

The explicit upper bound obtained above is enormous and is not
expected to be close to sharp. Several natural problems remain.

\begin{question}
For fixed $k$, what is the correct order of growth of $\mathcal{N}(k,n)$ as $n\to\infty$?
\end{question}

\begin{question}
Determine $\mathcal{N}(k,n)$ exactly in the first genuinely higher-dimensional cases, such as $(k,n)=(2,4)$ and $(2,5)$.
\end{question}

\section*{Declaration of competing interests}
The author declares that he has no known competing interests.

\section*{Funding}
The author received no specific funding for this work.

\section*{Acknowledgments} The author would like to thank Professors Liang Xiao, Qingchun Tian, and Binyong Xie for their invaluable support and assistance throughout this research, and also extends appreciation to the School of Mathematical Sciences at Peking University for providing a pleasant working environment.

\begin{flushright}
			\begin{minipage}{148mm}\sc\footnotesize
				J.\,L.: School of Mathematical Sciences, Peking University, Beijing, China \\
				{\it E-mail address}: \href{mailto:jialimath001@pku.org.cn}{{\tt jialimath001@pku.org.cn}} \vspace*{3mm}
			\end{minipage}
		\end{flushright}

\end{document}